\documentclass[graybox]{svmult}

\usepackage{mathptmx}       
\usepackage{helvet}         
\usepackage{courier}        
\usepackage{type1cm}        
\usepackage{makeidx}         
\usepackage{graphicx}        
\usepackage{multicol}        
\usepackage[bottom]{footmisc}

\makeindex             

\newcommand{\norm}[1]{\left\Vert#1\right\Vert}
\newcommand{\abs}[1]{\left\vert#1\right\vert}

\newcommand{\R}{\mathbf{R}}
\newcommand{\Z}{\mathbf{Z}}
\newcommand{\N}{\mathbf{N}}
\newcommand{\C}{\mathbf{C}}
\newcommand{\T}{\mathbf{T}}

\newcommand{\Q}{\mathbf{Q}}

\newcommand{\spec}{\mbox{Spec}}

\newcommand{\U}{{\mathcal U}}

\newcommand{\Emb}{\mbox{Emb}}

\newcommand{\Hom}{\mbox{Hom}}
\newcommand{\Fix}{\mbox{Fix}}

\begin{document}

\title*{Discrete Symmetric Planar Dynamics}
\author{B.\ Alarc\'on, S. B.\ S.\ D.\ Castro and I.\ S.\ Labouriau }
\institute{B.\ Alarc\'on \at Departamento de Matem\'aticas de la Universidad de Oviedo; Avda. Calvo Sotelo s/n; 33007 Oviedo; Spain. \email{alarconbegona@uniovi.es}
\and S. B.\ S.\ D.\ Castro \at Centro de Matem\'atica da Universidade do Porto and
Faculdade de Economia do Porto, Rua Dr. Roberto Frias, 4200-464 Porto, Portugal. \email{sdcastro@fep.up.pt}
\and  I.\ S.\ Labouriau  \at Centro de Matem\'atica da Universidade do Porto, Rua do Campo Alegre 687, 4169-007 Porto, Portugal. \email{islabour@fc.up.pt}}
%
%
\maketitle

\abstract*{We review previous results providing sufficient conditions to determine the global dynamics for equivariant maps of the plane with a unique fixed point which is also hyperbolic.
}

\abstract{We review previous results providing sufficient conditions to determine the global dynamics for equivariant maps of the plane with a unique fixed point which is also hyperbolic.
}

\section{Introduction}
The Discrete Markus-Yamabe Question  is a problem concerning discrete   dynamics, formulated in dimension $n$ by
Cima {\em et al.} \cite{Cima-Manosa} as follows:
\medbreak
\noindent{\em
{\bf [DMYQ($n$)]}
Let $f:\R^n\longrightarrow\R^n$ be a $C^1$ map such that $f(0)=0$ and for any $x\in\R^n$, $ Jf(x)$ has all its eigenvalues with modulus less than one.
Is it true that $0$ is a global attractor for the discrete dynamical system generated by $f$?
}
\medbreak

It is known that the answer  is affirmative in dimension $1$ and there are counter-examples  for dimensions higher than 2, see Cima {\em et al.} \cite{CvdEGHM} and van den Essen and Hubbers \cite{vdEH}.

In dimension 2, Cima {\em et al.} \cite{Cima-Manosa} prove that an affirmative answer is obtained when $f$ is a polynomial map, and provide a counter example which is a rational map.
After this, research on planar maps focused on the quest for minimal sufficient conditions under which the DMYQ has an affirmative answer.
Alarc\'on  {\em et al.} \cite{Alarcon} use the existence of an invariant embedded curve joining the origin to infinity to show the global stability of the origin.
Symmetry is a natural context for the existence of such a curve, and this led us to a symmetric approach to this problem and to the results in  \cite{AlarconDenjoy,SofisbeSzlenk,SofisbePolynomial,SofisbeGlobal,SofisbeSaddle} that we review in this article.

The present article studies maps $f$ of the plane which preserve symmetries described by the action of a compact Lie group.
In this setting we characterise the possible local dynamics near the unique fixed point of $f$, that we assume hyperbolic.
We establish for which symmetry groups local dynamics extends  globally.
For the remaining groups we present illustrative examples.

\section{Preliminaries} \label{secPre}

This section consists of definitions and known results about topological dynamics and equivariant theory. These are grouped in two separate subsections, which are elementary for readers in each field,
containing material from the corresponding sections of \cite{AlarconDenjoy,SofisbeSzlenk,SofisbePolynomial,SofisbeGlobal,SofisbeSaddle} and is included here for ease of reference.

\subsection{Topological Dynamics}

We consider planar topological embeddings, that is, continuous and injective maps defined in $\R^ 2$. The set of topological embeddings of the plane is denoted by $\Emb(\R^2)$.

Recall that for $f\in \Emb(\R^2)$ the equality $f(\R^2)=\R^2$ may not hold.
Since every map $f\in \Emb(\R^2)$ is open (see \cite{libroembeddings}), we will say that $f$ is a homeomorphism if $f$ is a topological embedding defined onto $\R^2$.
The set of homeomorphisms of the plane will be denoted by $\Hom(\R^2)$.
When $\mathcal{H}$ is  one of these sets  we denote by $\mathcal{H}^ {+}$ (and $\mathcal{H}^ {-}$) the subset of orientation preserving (reversing) elements of $\mathcal{H}$.

We denote by $\Fix(f)$  the set of  fixed points of  a continuous map $f: \R^2 \to \R^2$.

Let $\omega(p)$ be the set of points $q$ for which there is a sequence $n_j\to+\infty$ such that $f^{n_j}(q)\to p$.
If $f\in Hom(\R^2)$ then $\alpha(p)$ denotes the set $\omega(p)$ under $f^{-1}$.

Let $f\in \Emb(\R^2)$ and $p\in \R^2$.  We say that $\omega(p)=\infty$ if $\norm{f^n(p)}\to \infty$ as $n$ goes to $ \infty$. Analogously, if $f\in \Hom(\R^2)$, we say that $\alpha(p)=\infty$ if $\norm{f^{-n}(p)}\to \infty$ as $n$ goes to $ \infty$.


We say that a map $f\in \Emb(\R^2)$ is \emph{dissipative} if there exists a compact set $W\subset \R^2$ that is positively invariant and attracts uniformly all compact sets. This means that $f(W)\subset W$ and for each $x\in \R^2$,
$$
dist(f^n(x),W)\to 0, \quad \mbox{ as } \; n\to \infty
$$
uniformly on balls $\abs{x}\leq r$, $r>0$. Observe that in the case $f\in \Hom(\R^2)$ the dissipativity of $h$ means that $\infty$ is asymptotically stable for $f^{-1}$.

We say that $0\in \Fix(f)$ is a \emph{local attractor} if its basin of attraction $\U=\{p\in \R^ 2 : \omega(p)=\{0\}\}$ contains an open neighbourhood of $0$ in $\R^2$ and that $0$ is a \emph{global attractor} if $\U=\R^2$. The origin is a \emph{stable fixed point} if for every neighborhood $U$ of $0$ there exists another neighborhood $V$ of $0$ such that $f(V)\subset V$ and $f(V)\subset U$. Therefore, the origin is an \emph{asymptotically local (global) attractor} or a \emph{(globally) asymptotically stable fixed point} if it is a stable local (global) attractor. See \cite{Bhatia} for examples.

We say that $0\in \Fix(f)$ is a \emph{local repellor} if there exists a neighbourhood $V$ of $0$ such that $\omega(p)\notin V$ for all $0\neq p\in \R^ 2$ and a \emph{global repellor} if this holds for $V=\R^2$.

 We say that the origin is an \emph{asymptotically global repellor} if it is a  global repellor
 and, moreover, if for any neighbourhood $U$ of $0$ there exists another neighbourhood $V$ of $0$, such that,
 $V\subset f(V)$ and $V\subset f(U)$.

When the origin is a fixed point of a $C^ 1$-map of the plane we say the origin is a \emph{local saddle} if the two eigenvalues of $Df_0$, $\alpha, \beta$, are both real and verify $0<\abs{\alpha}<1<\abs{\beta}$. In case the two eigenvalues are strictly positive we say the origin is a direct saddle.
%
%
%
%
We say that the origin is a {\em global (topological) saddle} for a $C^1-$homeomorphism if additionally
its stable and unstable manifolds
$W^{s}(0,f)$, $W^{u}(0,f)$ are unbounded sets that do not accumulate on each other, except at $0$ and $\infty$,
and such that $$\R^2 \setminus (W^{s}\cup W^{u} \cup \{\infty\}) = U_1 \cup U_2 \cup U_3 \cup U_4,$$ where for all $i=1,...,4$ $\; U_i\subset \R^2$ is open connected and homeomorphic to $\R^2$ verifying:

\begin{itemize}
\item[i)] either $f(U_i)=U_i$ or there exists an involution $\varphi:\R^2 \to \R^2$ such that $(f \circ \varphi) (U_i)=U_i$
\item[ii)] for all $p\in U_i$ both $\norm{f^n(p)}\to \infty$ and $\norm{f^{-n}(p)}\to \infty$ as $n$ goes to $\infty$.
\end{itemize}

We say that $f\in \Emb(\R^2)$ has \emph{trivial dynamics} if  $\omega(p) \subset \Fix(f)$, for all $p\in \R^ 2$. Moreover, we say that a planar homeomorphism has trivial dynamics if both $\omega(p), \alpha(p) \subset \Fix(f)$, for all $p\in \R^ 2$.

Let $f: \R^N \rightarrow \R^N$ be a continuous map.
Let $\gamma : [0,\infty) \to \R^2$  be a topological
embedding of $[0,\infty) \,. \;$  As usual, we identify $\gamma$
with $\gamma\,([0,\infty))\,.$ We will say that $\gamma$ is an \emph{
$f$-invariant ray} if $\, \gamma(0)=(0,0)\,, \,$ $\,f(\gamma)\subset
\gamma \,, \,$ and $\lim_{t\to\infty}|\gamma(t)|=\infty$, where
$|\cdot| \,$ denotes the usual Euclidean norm.

\begin{proposition}[Alarc\'on {\em et al.} \cite{Alarcon}] \label{lemrayo} Let $f\in \Emb^ {+}(\R^2)$ be such that $\Fix(f)=\{0\}$. If there exists an $f$-invariant ray $\gamma$, then $f$ has trivial dynamics.
\end{proposition}

\begin{corollary} Let $f\in \Hom^{+}(\R^2)$  be such that $\Fix(f)=\{0\}$. If there exists an $f$-invariant ray $\gamma$, then for each $p\in\R^2$, as $n$ goes to $\pm \infty$, either $f^n(p)$ goes to $0$ or $\norm{f^n(p)}\to \infty$.
\end{corollary}

In order to explain the construction of examples in Section \ref{secDenjoy} we need to introduce the concept of prime end.

We say that $f:\R^2 \to \R^2$ is an \emph{admissible homeomorphism} if
$f$ is orientation preserving, dissipative and has an
asymptotically stable fixed point with proper and unbounded basin
of attraction $U\subset \R^2$. Note that $U$ is non empty, so the
proper condition follows when the fixed point is not a global
attractor. Since $f(U)=U$, we can obtain automatically the
unboundedness condition if we suppose that $f$ is area
contracting.

Let $f:\R^2 \to \R^2$ be an admissible homeomorphism and consider
 the compactification of $f$ to the Riemann sphere $f:\mathbf{S}^2 \to \mathbf{S}^2$.
 Hence $U\subset \mathbf{S}^2=\R^2 \cup
\{\infty\}$. A \emph{crosscut} $C$ of $U$ is an arc homeomorphic
to the segment $[0,1]$ such that $a,b\notin U$ and $\dot{C}= C
\setminus \{a,b\} \subset U$, where $a$ and $b$ are the extremes of
$C$. Every crosscut divides $U$ into two connected components
homeomorphic to the open disk $\D=\{z\in \C : \abs{z}<1\}$.

Let $x^*$ be a point in $U$. For convenience we will consider only
the crosscut such that $x^* \notin C$. We denote by $D(C)$ the
component of $U\setminus C$ that does not contain $x^*$. A
\emph{null-chain} is a sequence of pairwise disjoint crosscuts
$\{C_n\}_{n\in \N}$ such that $$\lim_{n\to \infty}
\mbox{diam}(C_n)=0 \mbox{ and } D(C_{n+1})\subset D(C_n),$$ where
$\mbox{diam}(C_n)$ is the diameter of $C_n$ on the Riemann sphere.

Two \emph{null-chains} are \emph{equivalent} $\{C_n\}_{n\in \N}
\sim\{C_n^*\}_{n\in \N}$ if given $m\in \N$
$$D(C_n)\subset D(C_m^*) \mbox{ and } D(C_n^*)\subset D(C_m),$$
for $n$ large enough. A \emph{prime end} is defined as a class of
equivalence of a null-chain and the space of prime ends is
$$\mathbf{P}=\mathbf{P}(U)=\mathcal{C}/\sim,$$
where $\mathcal{C}$ is the set of all
null-chains of $U$.

The disjoint union $U^*=U\cup \mathbf{P}$ is a topological space
homeomorphic to the closed disk $\bar{\D}=\{z\in \C :
\abs{z}\leq1\}$ such that its boundary is precisely $\mathbf{P}$.

It is well studied in \cite{pommerenke} that the Theory of Prime
Ends implies that an admissible homeomorphism $f$ induces an
orientation preserving homeomorphism $f^{*}:\mathbf{P} \to
\mathbf{P}$ in the space of prime ends. This topological space is
homeomorphic to the circle, that is $\mathbf{P} \simeq \T$, and
hence the rotation number of $f^{*}$ is well defined, say $\bar
\rho \in \T$. The \emph{rotation number} for an admissible
homeomorphism is defined by $\rho(f)=\bar \rho$.

\subsection{Equivariant Theory}

Let $\Gamma$ be a compact Lie group acting on $\R^2$, that is, a group which has the structure of a compact $C^{\infty}$-differentiable manifold such that the map $\Gamma \times \Gamma \to \Gamma$, $(x,y)\mapsto xy^{-1}$ is of class $C^{\infty}$. The folllowing definitions and results are taken from Golubitsky {\em et al.} \cite{golu2},  especially Chapter~XII,  to which we refer the reader interested in further detail.

We think of a group mostly through its action or representation on $\R^2$.
A {\em linear action} of $\Gamma$ on $\R^2$ is a continuous mapping
\begin{eqnarray*}
\Gamma \times \R^2 & \rightarrow & \R^2 \\
(\gamma , p ) & \mapsto & \gamma p
\end{eqnarray*}
such that, for each $\gamma \in \Gamma$ the mapping $\rho_{\gamma}$ that takes $p$ to $\gamma p$ is linear and, given $\gamma_1, \gamma_2 \in \Gamma$, we have $\gamma_1(\gamma_2 p)= (\gamma_1 \gamma_2) p$.
Furthermore the identity in $\Gamma$ fixes every point. The mapping $\gamma\mapsto \rho_{\gamma}$ is called the {\em representation} of $\Gamma$ and describes how each element of $\Gamma$ transforms the plane.

We consider only standard group actions and representations.
A representation of a group $\Gamma $ on a vector space $V$ is {\em absolutely irreducible} if the only linear mappings on $V$ that commute with $\Gamma $ are scalar multiples of the identity.

Given a map $f:\R^2\longrightarrow\R^2$,
we say that $\gamma \in \Gamma$ is a \emph{symmetry} of $f$ if $f(\gamma x)=\gamma f(x)$.
We define the \emph{symmetry group} of $f$ as the biggest closed subset of $GL(2)$ containing all the symmetries of $f$. It will be denoted by $\Gamma_f$.

We say that $f:\R^2\to \R^2$ is  \emph{$\Gamma$-equivariant} or that $f$ {\em commutes} with $\Gamma$ if
$$
f(\gamma x)=\gamma f(x) \quad \mbox{ for all }\quad \gamma \in \Gamma.
$$
It follows that
every map $f:\R^2\to \R^2$ is equivariant under the action of its symmetry group, that is, $f$ is $\Gamma_f$-equivariant.

Let  $\Sigma$ be a subgroup of $\Gamma$. The {\em fixed-point subspace} of $\Sigma$ is
$$
\Fix (\Sigma) =\{p\in \R^2: \sigma p=p \; \mbox{ for all } \; \sigma \in \Sigma\}.
$$ If $\Sigma$ is generated by a single element $\sigma \in \Gamma$, we write \emph{$\Fix\langle\sigma\rangle$}.

We note that, for each subgroup $\Sigma$ of $\Gamma$, $\Fix (\Sigma)$ is invariant by the dynamics of a $\Gamma$-equivariant map (\cite{golu2}, XIII, Lemma 2.1).

When $f$ is $\Gamma$-equivariant, we can use the symmetry to generalize information obtained for a particular point. This is achieved through the {\em group orbit} $\Gamma x$ of a point $x$, which is defined to be
$$
\Gamma x = \{ \gamma x: \; \; \gamma \in \Gamma\}.
$$

\begin{lemma}
Let $f:\R^2\to \R^2$ be $\Gamma$-equivariant and let $p$ be a fixed point of $f$. Then all points in the group orbit of $p$ are fixed points of $f$.
\end{lemma}

\begin{proof}
If $f(p)=p$ it follows that $f(\gamma p) = \gamma f(p) = \gamma p$, showing that $\gamma p$ is a fixed point of $f$ for all $\gamma \in \Gamma$.
\end{proof}

The relation between the group action and the Jacobian matrix of an equivariant map $f$ is obtained through the following

\begin{lemma}\label{LemmaJacobian}
Let $f: \; V \rightarrow V$ be a $\Gamma$-equivariant map differentiable at the origin. Then $Df(0)$, the Jacobian of $f$ at the origin, commutes with $\Gamma$.
\end{lemma}

\begin{proof}
Since $f$ is $\Gamma $-equivariant we have
$f(\gamma .v ) = \gamma f(v)$ for all $\gamma \in \Gamma$ and $v \in V$.
Differentiating both sides of the equality with respect to $v$, we obtain
$
Df(\gamma . v) \gamma = \gamma Df(v)
$
and, evaluating at the origin gives
$
Df(0) \gamma = \gamma Df(0).
$
\end{proof}

\section{Symmetries in the plane}

In this section, we describe the consequences for the local dynamics arising from the fact that a  map is equivariant under the action of a compact Lie group $\Gamma$. These are patent in the structure of the Jacobian matrix at the origin, obtained using Lemma~\ref{LemmaJacobian}.

Since every compact Lie group in $GL(2)$ can be identified with a subgroup of the orthogonal group $O(2)$, we need  only be concerned with the groups we list below.

\paragraph{Compact subgroups of $O(2)$}
\begin{itemize}
    \item  $O(2)$, acting on $\R^2 \simeq \C$ as the group generated by $\theta$ and $\kappa$ given by
    $$
    \theta . z = e^{i\theta} z, \quad \theta \in S^1 \qquad\mbox{ and   } \qquad \kappa . z=\bar{z}.
    $$
    \item  $SO(2)$, acting on $\R^2 \simeq \C$ as the group generated by $\theta$  given by
    $$
    \theta . z = e^{i\theta} z, \quad \theta \in S^1.
    $$
    \item  $D_n$, $n  \geq 2$, acting on $\R^2 \simeq \C$ as the finite group generated by $\zeta$ and $\kappa$ given by
    $$
    \zeta . z = e^{\frac{2\pi i}{n}} z  \qquad\mbox{ and   } \qquad \kappa . z=\bar{z}.
    $$
    \item  $\Z_n$, $n  \geq 2$, acting on $\R^2 \simeq \C$ as the finite group generated by $\zeta$ given by
    $$
    \zeta . z = e^{\frac{2\pi i}{n}} z.
    $$
    \item  $\Z_2(\langle\kappa\rangle)$, acting on $\R^2$ as
    $$
    \kappa . (x,y) = (x, -y).
    $$
\end{itemize}

Since most of our results depend on the existence of a unique fixed point for $f$,
it is worthwhile noting that
the group actions we are concerned with are such that $\Fix (\Gamma) = \{ 0 \}$.
Therefore, if $f$ is $\Gamma$-equivariant then $f(0)=0$.

If the representation is absolutely irreducible, we know that $Df(0)$ is a multiple of the identity and thus it  has one real eigenvalue of geometric multiplicity two. Therefore, the origin is locally either an attractor or a repellor.
We have the following

\begin{lemma}
The standard representation on $\R^2$ is absolutely irreducible for $O(2)$ and $D_n$ with $n \geq 3$ and for no other subgroup of $O(2)$.
\end{lemma}

\begin{proof}
The proof follows by direct computation.
\begin{itemize}
    \item  $O(2)$: the generators of this group are $\theta$ and $\kappa$ and it suffices to find the linear matrices that commute with both. A real matrix
    $$
    \left(\begin{array}{cc}
    a & b \\
    c & d
    \end{array} \right)
    $$
    commutes with $\kappa$ if and only if $b=c=0$. In order for such a matrix to commute with any rotation it must be
    $$
    \left( \begin{array}{cc}
    a & 0 \\
    0 & d
    \end{array} \right) \left(\begin{array}{cc}
    \cos{\theta} & -\sin{\theta} \\
    \sin{\theta} & \cos{\theta}
    \end{array} \right) = \left(\begin{array}{cc}
    \cos{\theta} & -\sin{\theta} \\
    \sin{\theta} & \cos{\theta}
    \end{array} \right) \left(\begin{array}{cc}
    a & 0 \\
    0 & d
    \end{array} \right)
    $$
    which holds when  $a=d$ or $\sin{\theta}=0$ for all $\theta \in S^1$.
    Therefore, the action of $O(2)$ is absolutely irreducible.
    \item  $SO(2)$: the elements of $SO(2)$ are rotation matrices which commute with any other rotation matrix, also non-diagonal ones.
    \item  $D_n$, $n  \geq 3$: see the proof for $O(2)$. In the last step, we must have $a=d$ or $\sin{2\pi i/n} = 0$ which is never satisfied for $n \geq 3$. Hence, the action is absolutely irreducible.
    \item  $\Z_n$, $n  \geq 3$: as for $SO(2)$, any rotation matrix commutes with the rotation of $2\pi /n$, including non-diagonal ones.
    \item  $\Z_2(\langle\kappa\rangle)$: see the proof for $\kappa \in O(2)$ to conclude that linear commuting matrices are diagonal but not necessarily linear multiples of the identity.
    \item  $\Z_2$: all linear maps commute with $-Id$.
    \item  $D_2=\Z_2 \oplus \Z_2(\langle\kappa\rangle)$: as above, $\Z_2$ introduces no restrictions and for commuting with $\kappa$ it suffices that the map is diagonal.
\end{itemize}
\end{proof}

The following result is then a straightforward consequence of the previous proof.

\begin{lemma}
The linear maps that commute with the standard representations of the subgroups of $O(2)$ are rotations and homotheties (and their compositions) for $SO(2)$ and $\Z_n$, $n \geq 3$, linear multiples of the identity for $O(2)$ and $D_n$, $n  \geq 3$, any linear map for $\Z_2$ and maps represented by diagonal matrices for the remaining groups.
\end{lemma}

\begin{proof}
The only linear maps that were not already explicitly calculated in the previous proof are those that commute with rotations. We have
$$
\left( \begin{array}{cc}
    a & b \\
    c & d
    \end{array} \right) \left(\begin{array}{cc}
    \cos{\theta} & -\sin{\theta} \\
    \sin{\theta} & \cos{\theta}
    \end{array} \right) = \left(\begin{array}{cc}
    \cos{\theta} & -\sin{\theta} \\
    \sin{\theta} & \cos{\theta}
    \end{array} \right) \left(\begin{array}{cc}
    a & b \\
    c & d
    \end{array} \right)
$$
if and only if either $\sin{\theta}=0$ for all $\theta \in S^1$ or $a=d$ and $b=-c$. Hence, the only maps commuting with either $SO(2)$ or $\Z_n$, $n \geq 3$, are rotations and homotheties and their compositions.
\end{proof}

With the results obtained so far, we are able to describe the Jacobian matrix at the origin for maps equivariant under each of the groups above.

\begin{proposition}[Proposition 2.3 in \cite{SofisbeGlobal}]
Let $f$ be a planar map differentiable at the origin. The admissible forms for the Jacobian matrix of $f$ at the origin are those given in Table \ref{Jacobian} depending on the symmetry group of $f$.
\end{proposition}

\begin{table}
\begin{center}
\begin{tabular}{ccr}
\hline {}&{}&{}\\
Symmetry group &   $Df(0)$ & hyperbolic local dynamics  \\  {}&{}&{}\\
\hline {}&{}&{}\\
$O(2)$ & $\left( \begin{array}{cc} \alpha & 0 \\ 0 & \alpha \end{array} \right)$ $\alpha \in \R$
&attractor / repellor \\ {}&{}&{}\\
\hline {}&{}&{}\\
$SO(2)$ &  $\left( \begin{array}{cc} \alpha & -\beta \\ \beta & \alpha \end{array} \right)$  $\alpha, \beta \in \R$ &attractor / repellor \\ {}&{}&{}\\
\hline {}&{}&{}\\
$D_n, \;n\geq 3$ &  $\left( \begin{array}{cc} \alpha & 0 \\ 0 & \alpha \end{array} \right)$ $\alpha \in \R$
&attractor / repellor \\ {}&{}&{}\\
\hline {}&{}&{}\\
$\Z_n, \;n\geq 3$ & $\left( \begin{array}{cc} \alpha & -\beta \\ \beta & \alpha \end{array} \right)$  $\alpha, \beta \in \R$ &attractor / repellor \\ {}&{}&{}\\
\hline {}&{}&{}\\
$\Z_{2}$ &  any matrix  &saddle / attractor / repellor \\ {}&{}&{}\\
\hline {}&{}&{}\\
$\Z_2(\langle\kappa\rangle)$ & $\left( \begin{array}{cc} \alpha & 0 \\ 0 & \beta \end{array} \right)$ $\alpha, \beta \in \R$ &saddle / attractor / repellor \\ {}&{}&{}\\
\hline {}&{}&{}\\
$D_{2}=\Z_2 \oplus \Z_{2}(\langle\kappa\rangle)$\quad & \quad $\left( \begin{array}{cc} \alpha & 0 \\  0 & \beta \end{array} \right)$ $\alpha, \beta \in \R$ \quad &saddle / attractor / repellor \\ {}&{}&{}\\
\hline
\end{tabular}
\end{center}
\caption{Compact subgroups of $O(2)$ and the admissible forms of the Jacobian at the origin of maps equivariant under the standard action of each group.
If in addition the Jacobian at the origin is hyperbolic, then this determines the local stability.}\label{Jacobian}
\end{table}

Furthermore, the symmetry constrains the normal form as described in \cite[Theorem 2.1]{SofisbePolynomial}
and in the next result, and its consequences for  the linear part of $f$ appear in Table~\ref{Jacobian}.

\begin{proposition}[Proposition 3.1 in \cite{SofisbePolynomial}]
Let $\Gamma$ be a compact Lie group acting on $\R^2$. Assume $\Gamma$ is the symmetry group of a polynomial map $f$.
\begin{itemize}
	\item[(i)]  If $\kappa \in \Gamma$, then $f$ does not answer the DMYQ($2$) in the affirmative unless $f$ is of the form:
	$$
	f(x,y)=\left(\begin{array}{cc}d_1&0\\0&d_2\end{array}\right)\left(\begin{array}{c}x\\ y\end{array}\right)+
	y^2p(y^2)\left(\begin{array}{c}1\\ 0\end{array}\right)\ .
	$$
	\item[(ii)]  If there is an element $\zeta  \in \Gamma$ of order $n \geq 3$, then $f$ does not answer the DMYQ($2$) in the affirmative unless $f$ is linear.
	Moreover,  the linear part of $f$ is either a homothety or a rotation matrix.
\end{itemize}
\end{proposition}

\section{Dynamics --- local to global}

\begin{figure}\begin{center}
a)\includegraphics[scale=0.40]{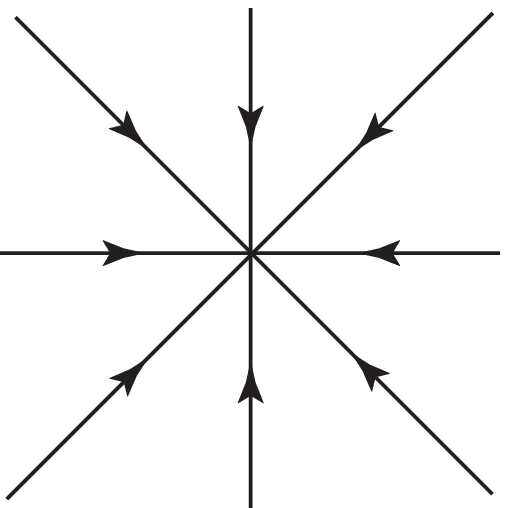}\qquad
b)\includegraphics[scale=0.40]{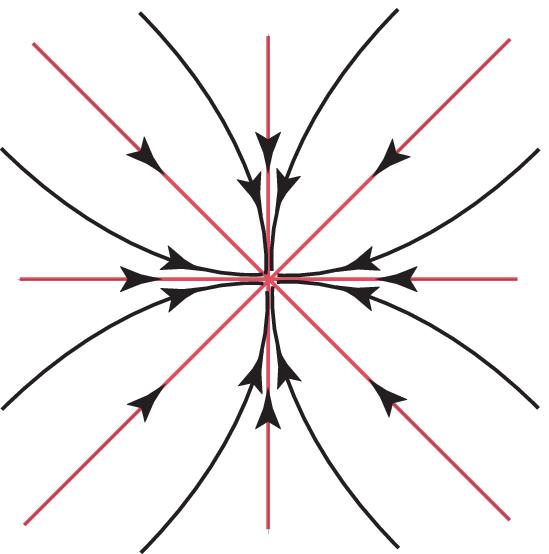}\qquad
c)\includegraphics[scale=0.40]{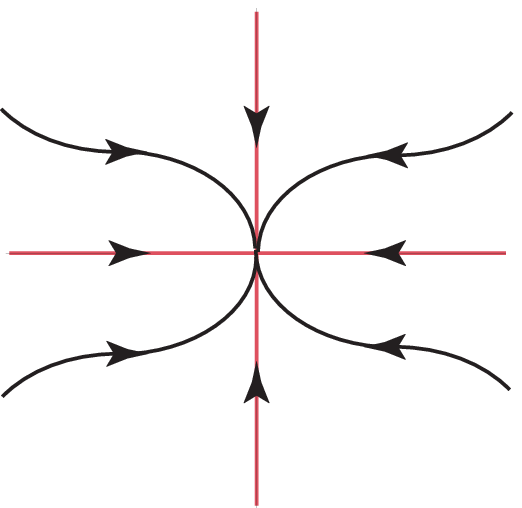}\qquad
d)\includegraphics[scale=0.40]{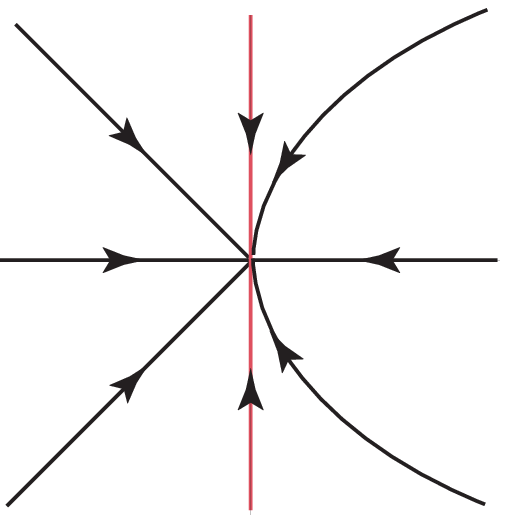}
\end{center}
\caption{
Local/global attractor with symmetry:
a) $O(2)$;
b) $D_4$ (without symmetries $\Z_8$ or $SO(2)$);
c) $D_2$ (without symmetry $D_4$);
d) $\Z_2\langle\kappa\rangle$  (without symmetry $D_4$)
\label{ComFlip}}\end{figure}

Figure~\ref{ComFlip} illustrates the dynamics near the origin  of equivariant maps for several symmetry groups.
A common feature of Figures~\ref{ComFlip} a)--d) is the existence of at least one symmetry axis.
This axis is the subspace fixed by a reflection and hence it is invariant under the dynamics.
Such a fixed-point subspace naturally contains an invariant ray
(see \cite[Lemma 3.3]{SofisbeGlobal}).
This allows us to use Proposition~\ref{lemrayo} to obtain the following results:

\begin{proposition}[Proposition 3.4 in \cite{SofisbeGlobal}] \label{proprayoS} Let $f\in \Emb(\R^2)$ have symmetry group $\Gamma$ with $\kappa \in \Gamma$, such that $\Fix(f)=\{0\}$. Suppose one of the following holds:

\begin{itemize}
\item[$a)$] $f\in \Emb^ {+}(\R^2)$ and $f$ does not interchange connected components of $\R^2 \setminus \Fix \langle\kappa\rangle$.
\item[$b)$]  $\Fix(f^ 2)=\{0\}$.
\end{itemize}
Then for each $p\in \R^2$ either $\omega(p)=\{0\}$ or $\omega(p)=\infty$.
\end{proposition}

The next example shows that assumption $b)$ in Proposition \ref{proprayoS} is necessary in the case where $f$ interchanges connected components of $\R^2 \setminus \Fix \langle\kappa\rangle$.

\paragraph{Example:} Consider the map $f: \R^2 \to \R^2$ defined by
$$
f(x,y)=\left(-ax^3+(a-1)x,-\frac{y}{2}\right) \quad 0<a<1.
$$
It is easily checked that $f$ has symmetry group $D_2$ and verifies (see Figure \ref{figureattractor}):
\begin{enumerate}
\item $f\in \Emb^{+}(\R^2)$ is an orientation-preserving diffeomorphism.
\item $\spec(f)\cap [0, \infty)=\emptyset$.
\item $0$ is a local hyperbolic attractor.
\item $\Fix(f^2)\neq \{0\}$.
\end{enumerate}
\medskip

\begin{figure}[htp]
\sidecaption
  \includegraphics[width=2in]{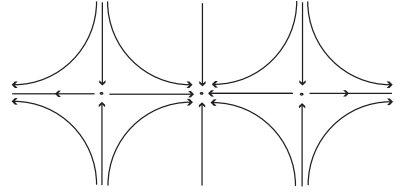}\\
  \caption{A local attractor which is not a global attractor due to the existence of periodic orbits.}\label{figureattractor}
\end{figure}

\begin{theorem}[Theorem 3.5  in \cite{SofisbeGlobal}] \label{proprayoS1} Let $f\in \Emb(\R^2)$ be dissipative with symmetry group $\Gamma$ with $\kappa \in \Gamma$ such that $\Fix(f)=\{0\}$. Suppose in addition that one of the following holds:

\begin{itemize}
\item[$a)$] $f\in \Emb^ {+}(\R^2)$ and $f$ does not interchange connected components of $\R^2 \setminus \Fix \langle\kappa\rangle$.
\item[$b)$]  There exist no $2$-periodic orbits.
\end{itemize}
Then $0$ is a global attractor.

\end{theorem}

\begin{corollary} [Corollary 3.6  in \cite{SofisbeGlobal}]Suppose the assumptions of Theorem \ref{proprayoS1}
are verified and $f$ is differentiable at $0$. If every eigenvalue of $Df(0)$ has norm strictly less than one, then $0$ is a global asymptotic attractor.
\end{corollary}

For  analogous results concerning a repellor see \cite{SofisbeGlobal}.

\begin{figure}
\sidecaption
a) \includegraphics[scale=0.35]{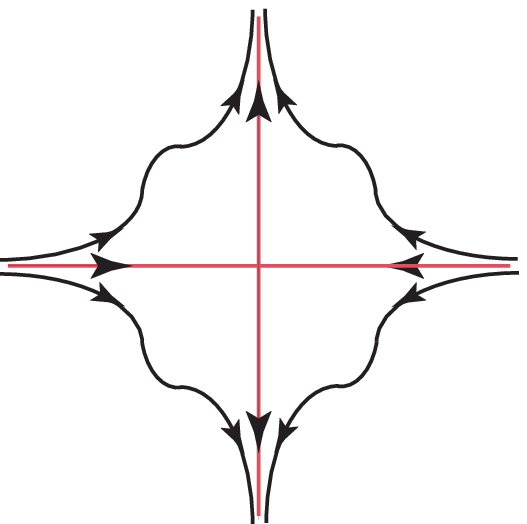}
\quad
b) \includegraphics[scale=0.35]{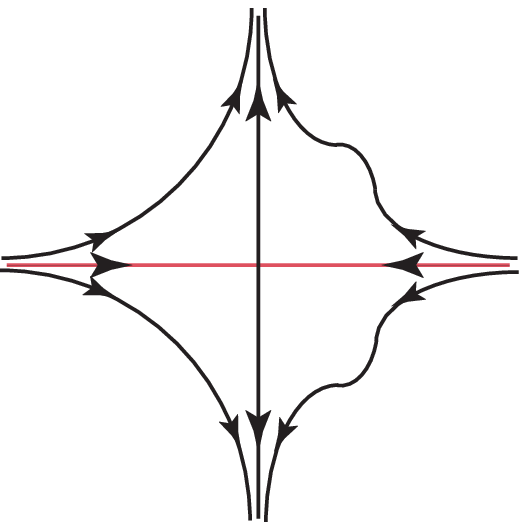}
\quad
c) \includegraphics[scale=0.35]{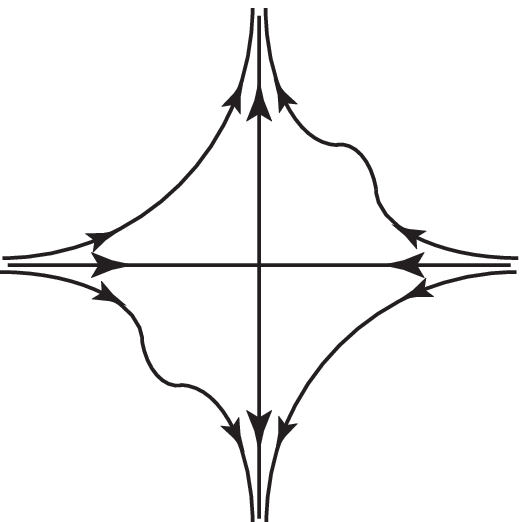}
\caption{
Local/global saddle with symmetry:
a) $D_2$;
b) $\Z_2(\kappa)$;
c) $\Z_2$.
\label{figSaddle}}\end{figure}

For the groups $\Z_{2}$, $\Z_2(\langle\kappa\rangle)$  and $D_{2}=\Z_2 \oplus \Z_{2}(\langle\kappa\rangle)$
the origin may also be a saddle as illustrated in Figure~\ref{figSaddle}.
For $D_{2}$, we have:

%
%

\begin{proposition} [\cite{SofisbeSaddle}]
Let $f\in \Hom(\R^ 2)$ be a $C^1$-homeomorphism with symmetry group $D_2$ such that $Fix(f)=\{0\}$. Suppose also that one of the following holds:

\begin{itemize}
\item[$a)$] $f$ is orientation preserving and $0$ is a direct saddle.
\item[$b)$] There exist no $2$-periodic orbits.
\end{itemize}

Then if $0$ is a local saddle, then $0$ is a global saddle.
\end{proposition}

In order to obtain a global saddle for $f$ with symmetry group either $\Z_{2}$ or $\Z_2(\langle\kappa\rangle)$,
we need the additional assumption that $f$ is a diffeomorphism, see \cite{SofisbeSaddle}.

\section{Strictly Local Dynamics} \label{secDenjoy}

Figure~\ref{Z4local} shows the local dynamics for maps equivariant under the action of groups
 that do not contain a reflection. These are $SO(2)$ and $\Z_n$. For these groups, local dynamics of attractor/repellor type does not necessarily extend to global dynamics, as we proceed to indicate.

We use examples referring to a local attractor, examples with a local repellor may be obtained considering $f^{-1}$.

\begin{figure}[hhh]
\sidecaption
a) \includegraphics[scale=0.25]{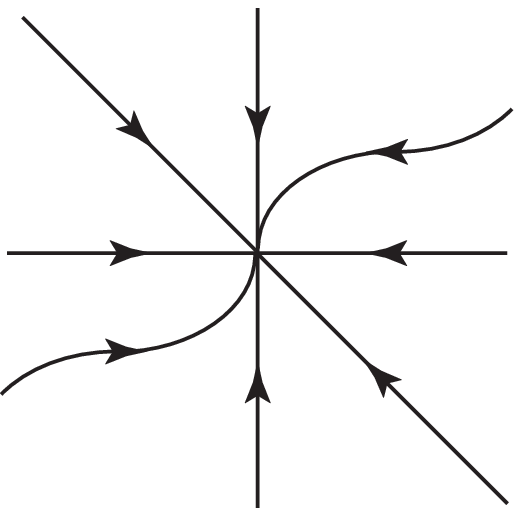}
\qquad
b) \includegraphics[scale=0.25]{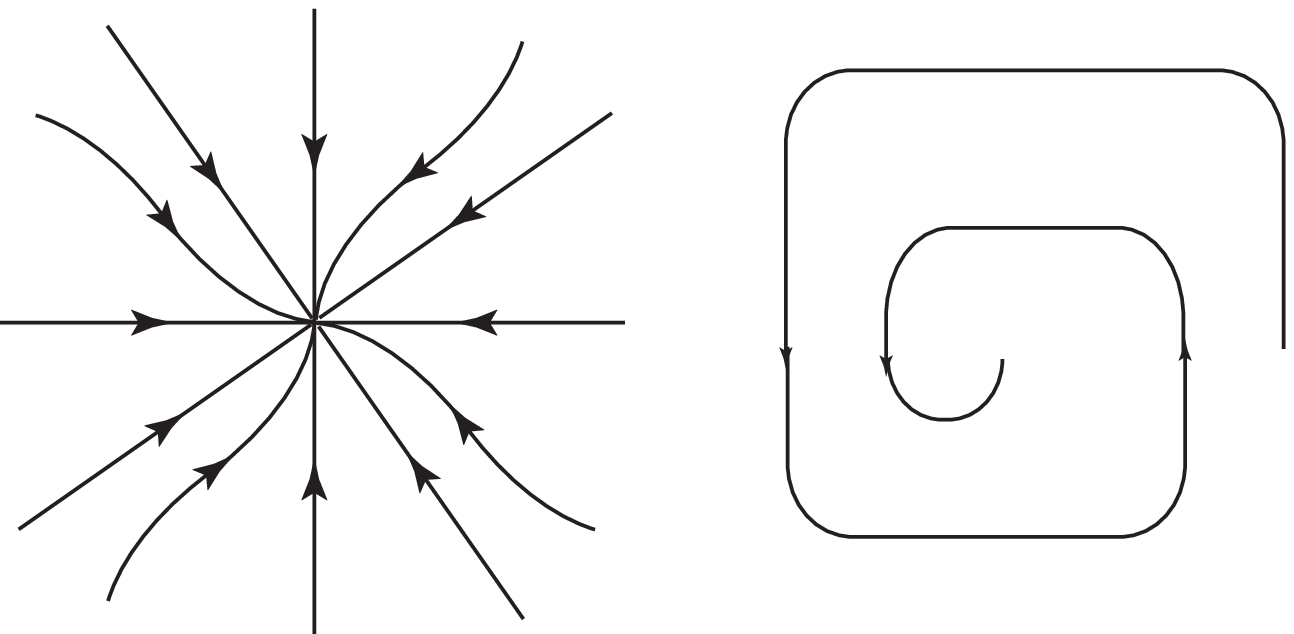}
\caption{
Local attractor with symmetry: a) $\Z_2$; b) $\Z_4$.
\label{Z4local}}
 \end{figure}

\bigbreak

The dynamics of  an $SO(2)$-symmetric embedding is mostly determined by  its radial component,
as  can be seen by writing $f$ in polar coordinates as $f(\rho, \theta)=(R(\rho, \theta), T(\rho, \theta))$.
It is easily shown that since $f$ is $SO(2)$-equivariant, the radial component  $R(\rho, \theta)$ only depends on $\rho$ and $R\in Emb(\R^+)$.
The fixed points of the radial component are invariant circles for $f$ hence knowledge about local dynamics does not contribute to the description of global dynamics unless $\Fix(R)= \{0\}$.

\bigbreak

The next two theorems show how a local attractor may be prevented from being a global attractor in a $\Z_n$-equivariant problem. Thus the examples in Figures~ \ref{Z4local} a) and b) may or may not extend to the whole plane.

\begin{theorem}[Theorem 3.1 in \cite{SofisbeSzlenk}] \label{exper} For each $n\ge 2$ there exists  $f:\R^2\to\R^2$ such
that:
\begin{enumerate}
\renewcommand{\theenumi}{\alph{enumi}}
\renewcommand{\labelenumi}{{\theenumi})}
\item\label{C1a} $f$ is a differentiable homeomorphism;
\item\label{C1b} $f$  has symmetry group $\Z_n$;
\item \label{C1E} $Fix(f)=\{0\}$;
\item\label{C1d} The origin is a local attractor;
\item \label{C3E} There exists a periodic orbit of minimal period $n$.
\end{enumerate}
\end{theorem}

The idea of the proof is to start with a $\Z_4$-equivariant example due to Szlenk (see \cite{Cima-Manosa}), sharing the same properties.
Each quadrant of the plane is invariant under the map $f_4$ of this example.
We deform the first quadrant into a sector of the plane, of angle $2\pi/n$ and then use the $\Z_n$ symmetry to cover the rest of the plane, as illustrated in Figure~\ref{FigSlzenkZn}.
The main difficulty is to prove that the result is a differentiable homeomorphism.

\begin{figure}
\sidecaption
\includegraphics[scale=.50] {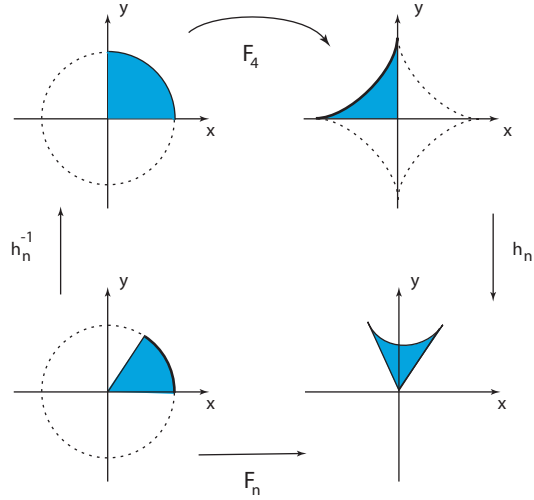}
\caption{Construction of the $\Z_n$-equivariant example $F_n$ in a fundamental domain of the
$\Z_n$-action, shown here for $n=6$.\label{FigSlzenkZn}}
\end{figure}

The $\Z_m$-equivariant homeomorphisms constructed in Theorem \ref{exper} have rotation number
$1/m$. So we might be led to think that the presence of the $\Z_m$-symmetry
implies that the rotation number of the homeomorphism should be
rational. One consequence would be that the asymptotically stable fixed point
is a global attractor if and only if there are no periodic points different from the fixed point.

The next result shows that this is false. We prove in \cite{AlarconDenjoy} the existence of $\Z_m$-equivariant and dissipative homeomorphisms in the plane with an asymptotically stable fixed point such that the induced map in the space of prime ends is conjugated to a
Denjoy map, which is also $\Z_m$-equivariant.
The idea is to construct $\Z_m$-equivariant Denjoy maps in the circle and then, in
the context of symmetry, to reproduce the construction used to prove the following:

\begin{proposition}[Proposition 2 in \cite{irrationalRotNumer}] Given $w \in (0,1)\setminus \Q$ and a Denjoy map $\phi$, there exists an admissible map $f$ with rotation number $\bar{w}$ and such that $f^{*}$
is topologically conjugated to $\phi$.
\end{proposition}

Observe that two admissible homeomorphisms $f_1, f_2$ with the same basin of
attraction $U$ verify that $$(f_1 \circ f_2)^*=f_1^*\circ f_2^*.$$

Let $f$ be an admissible homeomorphisms with basin of attraction U.
Suppose $f$ is $Z_m$-equivariant and $U$ is also invariant by $R_{\frac{1}{m}}$. Hence, the following holds: $$f^*\circ
R_{\frac{1}{m}}^*=R_{\frac{1}{m}}^*\circ f^*.$$

Since $R_{\frac{1}{m}}^*$ is a periodic homeomorphism of
$\T^1$ with rotation number $1/m$, then
$R_{\frac{1}{m}}^*$ is conjugated to the linear rotation
$R_{\frac{1}{m}}$ and $f^*$ is said to be $\Z_m$-\emph{equivariant in
the space of prime ends}.

\begin{theorem}[Theorem 4.2 in \cite{AlarconDenjoy}] \label{teoDenjoyZm} Given an irrational number $\tau\notin \Q$,
there exists a $\Z_m-$equivariant and admissible homeomorphism in
$\R^2$ with rotation number $\bar \tau \in \T$  and such that
induces a Denjoy map in the circle of prime ends which is also
$\Z_m-$equivariant.
\end{theorem}

 Hence, \cite{AlarconDenjoy} shows that for $\Z_m$-equivariant homeomorphisms one cannot guarantee that the rotation number is rational and proves the existence of $\Z_m$-equivariant homeomorphisms
with some complicated and interesting dynamical features.

\begin{acknowledgement}The research of all authors at Centro de Matem\'atica da Universidade do Porto (CMUP)
 had financial support from
 the European Regional Development Fund through the programme COMPETE and
 from  the Portuguese Government through the Funda\c c\~ao para
a Ci\^encia e a Tecnologia (FCT) under the project PEst-C/MAT/UI0144/2011.
B.\ Alarc\'{o}n was also supported by grant MICINN-12-MTM2011-22956 of the Ministerio de Ciencia e
Innovaci\'on (Spain).
\end{acknowledgement}

%
%

\begin{thebibliography}{99.}%
%

\bibitem{AlarconDenjoy} B. Alarc\'on. Rotation numbers for planar attractors of equivariant homeomorphisms,
2012. To appear in Topological Methods in Nonlinear Analysis.

\bibitem{SofisbeSzlenk} B. Alarc\'on, S.B.S.D Castro and I. Labouriau.  A local but not global attractor for a $\Z_n$-symmetric map,  {\em Journal of Singularities}, 6, 1--14, 2012.

\bibitem{SofisbePolynomial} B. Alarc\'on, S.B.S.D Castro and I. Labouriau.  The Discrete Markus-Yamabe Problem for Symmetric Planar Polynomial Maps. {\em Indagationes Mathematicae }, 23, 603--608, 2012.

\bibitem{SofisbeGlobal} B. Alarc\'on, S.B.S.D Castro and I. Labouriau.  Global Dynamics for Symmetric Planar Maps,  {\em Discrete \& Continuous Dyn. Syst.} - A, 37, 2241--2251, 2013.

\bibitem{SofisbeSaddle} B. Alarc\'on, S.B.S.D Castro and I. Labouriau. Global saddle for planar diffeomorphisms. {\em in preparation}.

\bibitem{Alarcon} B. Alarc\'on, V. Gu\'{\i}\~nez and C. Gutierrez. Planar Embeddings with a globally
attracting fixed point. {\em Nonlinear Anal.}, 69:(1), 140-150,
2008.

\bibitem{Bhatia} N.P. Bhatia and G.P. Szeg\"o. Stability Theory of Dynamical Systems. Springer-Verlag, New York, 2002.

\bibitem{Cima-Manosa} A. Cima, A. Gasull and F. Ma\~nosas. The Discrete Markus-Yamabe Problem, {\em Nonlinear Analysis}, 35, 343-354, 1999.

\bibitem{CvdEGHM}
{A. Cima, A. van den Essen, A. Gasull, E. Hubbers and F. Ma\~nosas}: A polynomial couterexample to the Markus-Yamabe conjecture. Advances in Mathematics 131, 453 --- 457, 1997.

\bibitem{vdEH}
{A. van den Essen and E. Hubbers}: A new class of invertible polynomial maps. Journal of Algebra 187, 214 --- 226, 1997

\bibitem{golu2} M. Golubitsky, I. Stewart and D.G. Schaeffer. Singularities and Groups in Bifurcation Theory Vol. 2. {\em Applied
Mathematical Sciences}, 69, Springer Verlag, 1985.

\bibitem{irrationalRotNumer} L. Hern\'andez-Corbato, R. Ortega and F. R. Ruiz del Portal.
Attractors with irrational rotation number, Math. Proc. Camb. Phil. Soc. 153, 59--77, 2012

\bibitem{libroembeddings} R. Ortega. Topology of the plane and periodic differential equations, {\em http://www.ugr.es/$\sim$ecuadif/fuentenueva.htm}

\bibitem{pommerenke} Ch. Pommerenke. Boundary Behaviour of Conformal Maps. {\em Springer-Verlag Berlin Heiderberg}, 1992.

\end{thebibliography}
%

\end{document}